\documentclass[11pt, a4paper]{amsart}
\usepackage{amssymb, array, amsmath, amscd, pdfpages, enumerate, amsthm, fixltx2e, setspace, hyperref}

\usepackage{xcolor}
\usepackage{mathtools}

\usepackage{ifthen}
\DeclareMathOperator{\re}{Re}			
\DeclareMathOperator{\im}{Im}			
\DeclareMathOperator{\Span}{span}		
\newcommand{\ran}{\mathcal{R}}						
\renewcommand{\ker}{\mathcal{N}}					

\DeclareMathOperator*{\dist}{dist}       

\newcommand*{\ldelim}[2]{\csname#1l\endcsname#2}   
\newcommand*{\rdelim}[2]{\csname#1r\endcsname#2}   
\newcommand*{\mdelim}[2]{\csname#1m\endcsname#2}   

\newcommand{\gl}{\lambda}
\newcommand{\gw}{\omega}
\newcommand{\gs}{\sigma}

\newcommand*{\C}{{\mathbb{C}}}     

\newcommand*{\R}{{\mathbb{R}}}     

\newcommand*{\Dom}{{\mathcal{D}}}   
\newcommand*{\Ran}{{\mathcal{R}}}   

\newcommand*{\abs} [1]{\lvert#1\rvert}
\newcommand*{\norm}[1]{\lVert#1\rVert}

\newcommand*{\set} [1]{\{#1\}}
\newcommand*{\iprod}[2]{\langle#1,#2\rangle}    

\newcommand*{\Abs}[2][default]{\ifthenelse{\equal{#1}{default}}{\left\lvert#2\right\rvert}{\ldelim{#1}{\lvert}#2\rdelim{#1}{\rvert}}}
\newcommand*{\Norm}[2][default]{\ifthenelse{\equal{#1}{default}}{\left\lVert#2\right\rVert}{\ldelim{#1}{\lVert}#2\rdelim{#1}{\rVert}}}

\newcommand*{\Iprod}[3][default]{\ifthenelse{\equal{#1}{default}}{\left\langle#2,#3\right\rangle}{\ldelim{#1}{\langle}#2,#3\rdelim{#1}{\rangle}}}
\newcommand*{\Dualpair}[3][default]{\ifthenelse{\equal{#1}{default}}{\left\langle#2,#3\right\rangle}{\ldelim{#1}{\langle}#2,#3\rdelim{#1}{\rangle}}}

\newcommand*{\Set}[2][default]{\ifthenelse{\equal{#1}{default}}{\left\{#2\right\}}{\ldelim{#1}{\{}#2\rdelim{#1}{\}}}}

\newcommand*{\dda}[3][1]{\ifthenelse{\equal{#1}{1}}{\frac{d#3}{d#2}}{\frac{d^{#1}#3}{d#2^{#1}}}}
\newcommand*{\ddb}[2][1]{\ifthenelse{\equal{#1}{1}}{\frac{d}{d#2}}{\frac{d^{#1}}{d#2^{#1}}}}
\newcommand*{\pd}[3][1]{\ifthenelse{\equal{#1}{1}}{\frac{\partial{#2}}{\partial{#3}}}{\frac{\partial^{#1}{#2}}{\partial#3^{#1}}}}
\newcommand*{\pdb}[2][1]{\ifthenelse{\equal{#1}{1}}{\frac{\partial}{\partial{#2}}}{\frac{\partial^{#1}}{\partial#2^{#1}}}}

\newcommand*{\Lp}[1][p]{L^{#1}}

\newcommand{\pmat}[1]{\begin{pmatrix}#1\end{pmatrix}}

\newcommand{\eq}[1]{\begin{align*}#1\end{align*}}
\newcommand{\eqn}[1]{\begin{align}#1\end{align}}

\newcommand{\hiddenproof}[2]{
\ifthenelse{#1}{Hidden}{#2}
}

\newcommand{\longcomm}[1]{}

\usepackage{datetime}
\newdateformat{findate}{\THEDAY.\THEMONTH.\THEYEAR}

\renewcommand{\ran}{\mathrm{Ran}}	
\renewcommand{\Ran}{\mathrm{Ran}}	
\renewcommand{\ker}{\mathrm{Ker}}
\newcommand{\Ker}{\mathrm{Ker}}

\renewcommand{\Dom}{D}
\renewcommand{\gw}{s}

\newcommand{\Sg}[1][T]{(#1(t))_{t\geq 0}}

\newcommand{\RR}{\mathbb{R}}
\newcommand{\CC}{\mathbb{C}}
\newcommand{\ZZ}{\mathbb{Z}}

\renewcommand{\Dom}{D}

\newtheorem{thm}{Theorem}[section]
\newtheorem{prp}[thm]{Proposition}
\newtheorem{lem}[thm]{Lemma}

\theoremstyle{definition}
\newtheorem{rem}[thm]{Remark}

\numberwithin{equation}{section}

\newcommand{\snorm}[1]{|\hspace{-.012in}|\hspace{-.012in}|#1|\hspace{-.012in}|\hspace{-.012in}|}

\renewcommand{\v}{u}
\newcommand{\vv}{v}
\newcommand{\f}{w}
\newcommand{\vs}{f}
\newcommand{\vvs}{g}
\newcommand{\fs}{h}

\newcommand{\sql}{\sqrt{\gl}}

\newcommand{\sqlwi}{\sqrt{is}}

\newcommand{\cv}[1][\gl]{a(#1)}
\newcommand{\cf}[1][\gl]{b(#1)}

\newcommand{\Ivs}[2][s]{U_{#1}(#2)}
\newcommand{\Ivc}[2][s]{U'_{#1}(#2)}
\newcommand{\Ifs}[2][s]{W_{#1}(#2)}
\newcommand{\Ifc}[2][s]{W'_{#1}(#2)}

\newcommand{\sprod}[2]{\left(#1\,\middle|\,#2\right)}

\begin{document}

\title[Optimal energy decay in a  coupled wave-heat system]{Optimal energy decay in a one-dimensional coupled wave-heat system}

\author{Charles Batty}
\address{St John's College, St Giles, Oxford\;\;OX1 3JP, United Kingdom}
\email{charles.batty@sjc.ox.ac.uk}

\author{Lassi Paunonen}
\address{Department of Mathematics, Tampere University of Technology, P.O.\ Box 553, 33101 Tampere, Finland}
\email{lassi.paunonen@tut.fi}

\author{David Seifert}
\address{St John's College, St Giles, Oxford\;\;OX1 3JP, United Kingdom}
\email{david.seifert@sjc.ox.ac.uk}

\begin{abstract}
We study a simple one-dimensional coupled wave-heat system and obtain a sharp estimate for the rate of energy decay of classical solutions. Our approach is based on the asymptotic theory of $C_0$-semigroups and in particular on a result due to Borichev and Tomilov \cite{BT10}, which reduces the problem of estimating the rate of energy decay to finding a growth bound for the resolvent of the semigroup generator. This technique not only leads to an optimal result, it is also simpler than the methods used by other authors in similar situations.
\end{abstract}

\thanks{This work was carried out while the second author visited Oxford in March and April 2015. The visit was supported by the EPSRC grant  EP/J010723/1 held by the first author and Professor Y.\ Tomilov (Warsaw). The second author would also like to thank the  Vilho, Yrj\"o and Kalle V\"ais\"al\"a Foundation for financial support.}

\maketitle

\section{Introduction}
\label{sec:Intro}

We study the following one-dimensional coupled wave-heat system:
\eqn{
\label{eq:problem}
\left\{
\begin{array}{ll@\qquad r}
  \multicolumn{2}{l}{\v_{tt}(\xi,t) =\v_{\xi\xi}(\xi,t),} 
  & \xi\in(-1,0), ~t>0, \\[1ex]
  \multicolumn{2}{l}{\f_{t}(\xi,t) = \f_{\xi\xi}(\xi,t),} &   \xi\in(0,1), ~t>0,\\[1.5ex]
    \v_\xi(-1,t)=0, & \f(1,t)=0, & t>0,\\[1ex]
    \v_t(0,t)=\f(0,t), & \v_\xi(0,t)=\f_\xi(0,t), & t>0,\\[1.5ex]
    \v(\xi,0)=\v(\xi), & \v_t(\xi,0)=v(\xi), &  \xi\in(-1,0),
  \\[1ex]
  \multicolumn{2}{l}{\f(\xi,0)=\f(\xi),} & \xi\in(0,1),
\end{array}
\right.
}
where the initial data satisfy $u\in H^1(-1,0)$, $v\in L^2(-1,0)$ and $w\in L^2(0,1)$. Models of this type have received considerable attention in recent years, especially in the higher-dimensional setting. The main motivation comes from the study of so-called fluid-structure models, in which the non-linear elasticity equation is coupled with the Navier-Stokes equations; see for instance the survey article \cite{AvaTri07a}. The above model can be viewed as a simple linearised version of such a model which however preserves the feature of a hyperbolic equation being coupled with a parabolic equation. 

The primary aim in such models is to obtain quantitative estimates for the rate of energy decay of a given solution. Given a vector $x=(u,v,w)$ of initial data as above, the energy of the solution corresponding to these initial data is defined as 
$$E_x(t)=\frac{1}{2}\int_{-1}^1|u_\xi(\xi,t)|^2+|u_t(\xi,t)|^2+|w(\xi,t)|^2\, d\xi,\quad t\ge0,$$
where all functions have been extended by zero in $\xi$ to the interval $(-1,1)$.  A simple calculation shows that
$$E_x'(t)=-\int_0^1|w_\xi(\xi,t)|^2\,d\xi,\quad t\ge0,$$
provided the solution is sufficiently regular, and hence the energy of any such solution is non-increasing in time. The aim of this paper is to establish an optimal estimate on the rate at which this decay occurs.

A higher-dimensional version of the above problem was studied in \cite{ZhaZua07}, showing that $E_x(t)=O(t^{-1/3})$ as $t\to\infty$, and this rate was improved in \cite{Duy07} to $t^{-2+\delta}$ for any $\delta>0$. 
A rate of $t^{-1}$ is established for a closely related model in \cite{AvaTri13}, where the authors also comment on the optimality of the decay rates for the higher-dimensional systems; see \cite[Rem.\ 1.3]{AvaTri13}.
Further related results may be found 
in~\cite{Ava07,
AvaDvo08,
RauZha05,
VazZua03}.
As the main result of this paper we show that the energy of any classical solution to \eqref{eq:problem} 
satisfies 
$$E_x(t)=o\big(t^{-4}\big),\quad t\to\infty.$$
Moreover, we prove that this rate is optimal. A similar result was obtained in \cite{ZhaZua04} for a similar one-dimensional problem in which the wave part satisfies a Dirichlet boundary condition rather than a Neumann boundary condition at 
 $\xi=-1$. Our method is based on the semigroup approach used in all of the aforementioned works; see also \cite{AvaTri07, AvaTri09}. However, rather than estimating the norm of the semigroup directly we follow \cite{AvaTri13} in using a recent result in the abstract theory of $C_0$-semigroups obtained by Borichev and Tomilov in \cite{BT10}, which makes it possible to deduce sharp rates of energy decay from appropriate growth bounds on the norm of the resolvent of the generator; see also \cite{ BatChi14, BD08}. This method not only leads to an optimal decay estimate for the problem under consideration, it also leads to a relatively concise argument. In particular, the present method appears to be significantly simpler than the method used in the  Dirichlet case in \cite{ZhaZua04}, which is based on  the theory of Riesz spectral operators and a very detailed spectral analysis.

The paper is set out as follows.  In Section~\ref{sec:sggen} we show that our problem is well-posed in the sense of $C_0$-semigroups and we describe the spectrum of the infinitesimal generator. In Section~\ref{sec:res} we turn to resolvent bounds. We first establish in Theorem~\ref{thm:upperbd} an upper bound for the norm of the resolvent operator along the imaginary axis. This is the most technical part of the paper. We then go on to show in Theorem~\ref{thm:opt} that the upper bound is optimal. In Section~\ref{sec:energy} we apply the Borichev-Tomilov result to deduce from these resolvent bounds an optimal estimate for the rate of energy decay of smooth solutions. Finally, in Section~\ref{sec:Dirichlet} we conclude with some comments on the Dirichlet case addressed in \cite{ZhaZua04}, explaining how the present approach can be adapted to that setting without difficulty.

The notation used is standard throughout. In particular, given a closed operator $A$ on a Banach space $X$, which will always be assumed to be complex, we denote its domain by $D(A)$, its kernel by $\ker(A)$ and its range by $\ran(A)$. The spectrum of $A$ is denoted by $\sigma(A)$, and its resolvent set by $\rho(A)$. Given $\lambda\in\rho(A)$, we write $R(\lambda,A)$ for the resolvent operator $(\lambda-A)^{-1}$. 
Given two functions $f,g:(0,\infty)\to\RR_+$, we write $f(t)=O(g(t))$, $t\to\infty$, to indicate that $f(t)\le Cg(t)$ for some constant $C>0$ and all sufficiently large $t>0$. If $g(t)>0$ for all  $t>0$, we write $f(t)=o(g(t))$, $t\to\infty$, if $f(t)/g(t)\to0$ as $t\to\infty$. The case of real-valued functions defined on the whole of $\RR$ is treated in an analogous way.
For real-valued quantities $p$ and $q$, we occasionally use the notation $p\lesssim q$ to indicate that $p\le Cq$ for some constant $C>0$ which is independent of all the parameters that are free to vary in a given situation.
We let $\CC_-$ denote the open left half-plane $\{\lambda\in\CC:\re\lambda<0\}$.

\section{Well-posedness and  properties of the semigroup}
\label{sec:sggen}

We begin by recasting \eqref{eq:problem} as an abstract Cauchy problem. Consider the two Hilbert spaces $X=H^1(-1,0)\times L^2(-1,0)\times L^2(0,1)$ and  $Y=H^2(-1,0)\times H^1(-1,0)\times H^2(0,1)$, both with their natural norms.
The operator $A$ defined by $Ax=(v,u'',w'')$ for $x=(u,v,w)$ in the domain
$$D(A)=\big\{(u,v,w)\in Y: u'(-1)=w(1)=0, ~v(0)=w(0), ~u'(0)=w'(0)\big\}$$
is closed and densely defined. If we let $z(t)=(u(\cdot,t),u_t(\cdot,t),w(\cdot,t))$ for $t\ge0$, then the coupled wave-heat system \eqref{eq:problem} can be rewritten in the form

\eqn{
\label{eq:CP}
\begin{cases}
z'(t)=Az(t), &t\ge0,\\
z(0)=x,
\end{cases}
}
where $x\in X$. We show below, in Theorem~\ref{thm:sggeneration}, that $A$ generates a uniformly bounded $C_0$-semigroup $\Sg$
on $X$. In particular, the unique solution of \eqref{eq:CP} is given by $z(t)=T(t)x$ for $t\ge0$. This solution in general satisfies \eqref{eq:problem} only in the so-called mild sense, and it is a solution in the classical sense precisely when $x\in D(A)$; see for instance \cite{ABHN11} for details on the theory of $C_0$-semigroups. It will be for such classical solutions that we eventually, in Section~\ref{sec:energy}, obtain a uniform rate of energy decay in the sense described in Section~\ref{sec:Intro}.

\begin{thm}
  \label{thm:sggeneration}
The operator $A$ generates a uniformly bounded $C_0$-semigroup $\Sg$ on $X$. Moreover, the spectrum $\sigma(A)$ of $A$ consists of isolated eigenvalues and satisfies $\sigma(A)\cap i\RR=\{0\}$ and, in fact, 
\eq{
\sigma(A)=\big\{\lambda\in\CC: \sql \cosh(\lambda)\cosh(\sql)+ \sinh(\lambda)\sinh(\sql)=0\big\}.
}
\end{thm}

We prove this theorem in a sequence of smaller results.

\begin{prp}
  \label{prp:kerA}
The spectrum $\sigma(A)$ of $A$ consists of isolated eigenvalues and satisfies $\sigma(A)\cap i\RR=\{0\}$. Moreover, $ \ker(A) = \Span \set{(1,0,0)}$.
\end{prp}

\begin{proof}

By the Rellich-Kondrachov theorem $Y$ embeds compactly into $X.$ Since the resolvent operator $R(\lambda,A)$, for any $\gl\in\rho(A)$, maps $X$ isomorphically onto $D(A)$ endowed with the graph norm and this space embeds continuously into $Y$, $A$ has compact resolvent. In particular, the spectrum of $R(\gl,A)$ consists only of eigenvalues of finite multiplicity whose only possible accumulation point is the origin. By the spectral mapping theorem for the point spectrum, $\sigma(A)$ consists only of eigenvalues of finite multiplicity whose only possible accumulation point is at infinity.

  For elements $x=(u,v,w)$ and $y=(f,g,h)$ of $X$, let
\eq{
  \sprod{x}{y}
  := \iprod{\v'}{\vs'}_{\Lp[2]} + \iprod{\vv}{\vvs}_{\Lp[2]} + \iprod{\f}{\fs}_{\Lp[2]} .
  }
  Suppose that $s\in \RR$ with 
  $s\neq 0$ and that $x=(\v,\vv,\f)\in \Ker(is-A)$, i.e.\ $x\in D(A)$ and 
  $(is-A)x=0$. 
A routine calculation using integration by parts gives 
    \eq{
   0= \re \sprod{(is-A)x}{x}= \re \sprod{-Ax}{x} =\norm{\f'}_{\Lp[2]}^2.
  }
  Since $\f\in H^2(0,1)$ and  $\f(1)=0$ this implies that $\f= 0$.
  The equation $(is-A)x=0$ now reduces to
  \begin{subequations}
\label{eq:KerBVPs}
\eqn{
\label{eq:KerDE1}
\hspace{2cm} \v''(\xi)&=-s^2\v(\xi), \hspace{-0.2cm} & \hspace{-2.5cm} \xi\in(-1,0), \hspace{1.5cm}\\
\hspace{2cm}     \vv(\xi)&=is u(\xi), \hspace{-0.2cm}  & \hspace{-2.5cm} \xi\in(-1,0), \hspace{1.5cm}\\
& \hspace{-.2cm}  \v'(-1)=u'(0)=v(0)=0.
}
\end{subequations}
      The general solution of \eqref{eq:KerDE1} with the boundary condition $\v'(-1)=0$ takes the form  
      $$\v(\xi) = \cv[s] \cos(s(\xi+1)),\quad \xi\in[-1,0],$$
      where $\cv[s]\in\C$, and the  boundary conditions at $\xi=0$ give 
      $$\cv[s] \sin(s)=is\, \cv[s] \cos(s)=0.$$
       Since $s\neq 0$ these two equations imply that $\cv[s] =0$. Thus the solution of \eqref{eq:KerBVPs} is $u=v=0$, and hence $x=0$. It follows that $\ker(is-A)=\{0\}$ for $s\ne0$, so that $\gs_p(A)\cap i\R\subset \set{0}$.
  
It remains to consider the case $s=0$. Certainly $(c,0,0)\in \ker(A)$ for any $c\in\CC$.
On the other hand, if $x=(\v,\vv,\f)\in \ker(A)$, then  $\f=0$ as above and
\begin{eqnarray*}
  &
  \begin{aligned}
    \v''(\xi)&=0, \qquad &\xi\in(-1,0),\\
    \vv(\xi)&=0,  &\xi\in(-1,0),
  \end{aligned}
  \\
  &
  \v'(-1)=
  \vv(0)=
  \v'(0)=0,
\end{eqnarray*} 
      whose solutions are of the form $u=c$, for some constant $c\in \C$, and $v=0$. Thus $0\in \gs_p(A)$ and $ \ker(A) = \Span \set{(1,0,0)}$, as required.
\end{proof}

\begin{prp}
  \label{prp:Aeigvals}
The eigenvalues of $A$ are precisely those points $\lambda\in\CC$ which satisfy
\eqn{
\label{eq:det=0}
 \sql \cosh(\gl)\cosh(\sql) +  \sinh(\gl)\sinh(\sql)  =0.
}
\end{prp}

\begin{proof}
Suppose that  $\gl\neq 0$ and let
$x=(\v,\vv,\f)\in \ker(\lambda-A)$.
Then 
\begin{subequations}
\label{eq:EVBVPs}
\eqn{
\label{eq:EVDE1}
\hspace{2cm} \v''(\xi)&=\gl^2\v(\xi), \hspace{-0.2cm} & \hspace{-3.3cm} \xi\in(-1,0), \hspace{3.5cm}\\
\hspace{2cm}     \vv(\xi)&=\gl u(\xi), \hspace{-0.2cm}  & \hspace{-4cm} \xi\in(-1,0), \hspace{3.5cm}\\
\label{eq:EVDE3}
\hspace{2cm}     \f''(\xi)&=\gl w(\xi), \hspace{-0.2cm} & \hspace{-4cm} \xi\in(0,1), \hspace{3.5cm} \\
& \hspace{-2.2cm}  \v'(-1)=  \f(1)=0, ~
  \vv(0)=\f(0), ~
  \v'(0)=\f'(0).
}
\end{subequations}
The
 general solution of \eqref{eq:EVDE1}
with the boundary condition $\v'(-1)=0$ is 
  \eqn{
  \label{eq:usol0}
\v(\xi) = \cv \cosh(\gl(\xi+1)),\quad \xi\in[-1,0],
  } 
  where $\cv\in \C$.
  Moreover,
the solution of \eqref{eq:EVDE3}
with the boundary condition $\f(1)=0$ is given by
  \eqn{
  \label{eq:wsol0}
  \f(\xi)=  \cf \sinh(\sql(\xi-1)),\quad \xi\in[0,1],
  }
  where $\cf\in \C$.
Using \eqref{eq:usol0} and \eqref{eq:wsol0}, 
the coupling condition $\f(0)=\vv(0)=\gl \v(0)$ becomes
\eq{
    \gl \cv \cosh(\gl)+  \cf \sinh(\sql)  =0.
}
Similarly, the condition $\f'(0)=\v'(0)$ can be written in the form
\eq{
 \gl\cv\sinh(\gl) -\sql\cf \cosh(\sql)  = 0.
}
Thus $\ker(\lambda-A)=\{0\}$ if and only if the equation
\eq{
\pmat{
\gl  \cosh(\gl) &
 \sinh(\sql)\\[1.5ex] 
\gl\sinh(\gl) & -\sql\cosh(\sql) }
\pmat{\cv\\\cf}
=0
}
 has only the trivial solution $\cv=\cf=0$. In other words,  $\lambda\in\sigma(A)$ if and only if the matrix on the left-hand side has zero determinant. For $\lambda\ne0$, this is equivalent to \eqref{eq:det=0}.
\end{proof}

\begin{prp}
  \label{prp:Xdecomp}
  The space $X$ splits as a topological direct sum 
  $$X=\Ran(A)\oplus\ker(A),$$
   and on the closed subspace $\Ran (A)$ the norm of $X$ is equivalent to the norm $\snorm{\cdot}$ given by 
$$\snorm{(u,v,w)}=\left(\|u'\|_{L^2}^2+\norm{v}_{L^2}^2+\norm{w}_{L^2}^2\right)^{1/2}.$$
\end{prp}

\begin{proof}
 We begin by showing that $\Ran(A)=\ker(\phi)$, where $\phi:X\to\CC$ is the bounded linear functional that maps $x=(u,v,w)\in X$ to 
 \eqn{
\label{eq:phi}
\phi(x)=u(0)+\int_{-1}^0v(\xi)\,d\xi+\int_0^1(1-\xi)w(\xi)\, d\xi.
}
This shows in particular that $\Ran(A)$ is closed and that the projection $x\mapsto(\phi(x),0,0)$ onto  $\ker(A)$ along $\ran(A)$  is bounded. 

Suppose that $y=(\vs,\vvs,\fs)\in \ran(A)$. Then there exists $x=(\v,\vv,\f)\in D(A)$ such that $Ax=y$, which leads to the equations
\begin{subequations}
\label{eq:BVPs}
\eqn{
\label{eq:DE1}
\hspace{2cm} \v''(\xi)&=\vvs(\xi), \hspace{-0.2cm} & \hspace{-4cm} \xi\in(-1,0), \hspace{3.5cm}\\
\hspace{2cm}     \vv(\xi)&=\vs(\xi), \hspace{-0.2cm}  & \hspace{-4cm} \xi\in(-1,0), \hspace{3.5cm}\\
\label{eq:DE3}
\hspace{2cm}     \f''(\xi)&=\fs(\xi), \hspace{-0.2cm} & \hspace{-4cm} \xi\in(0,1), \hspace{3.5cm} \\
& \hspace{-2.2cm}  \v'(-1)=  \f(1)=0, ~
  \vv(0)=\f(0), ~
  \v'(0)=\f'(0).
}
\end{subequations}
A direct computation shows that  the solutions of \eqref{eq:DE1}--\eqref{eq:DE3}, subject to the boundary conditions $\v'(-1)=0$ and $\f(1)=0$, are of the form
\eq{
\v(\xi) &= a + \int_{-1}^\xi(\xi-r)\vvs(r)\,dr,\\[1ex]
\vv(\xi) &= \vs(\xi),\\
\f(\xi)  
&= b(1-\xi) +  \int_{\xi}^1 (r-\xi)\fs(r)\,dr,
}
where $a,b\in\CC$ are constants. It is now easy to verify that the coupling conditions $\vv(0)=\f(0)$ and $\v'(0)=\f'(0)$   can both be satisfied only if $\phi(y)=0$, which shows that $\Ran(A)\subset\ker(\phi)$. 

Conversely, if $y=(\vs,\vvs,\fs)\in X$ is such $\phi(y)=0$, then we can define $x=(\v,\vv,\f)\in Y$ by
\eq{
\v(\xi) &= \int_{-1}^\xi(\xi-r)\vvs(r)\,dr,\\[1ex]
\vv(\xi) &= \vs(\xi),\\
\f(\xi)  
&= (1-\xi)\left(\vs(0)-\int_0^1 r\fs(r)\,dr  \right) +  \int_{\xi}^1 (r-\xi)\fs(r)\,dr.
}
A direct computation shows that $x\in \Dom(A)$ and that $u,v,w$  satisfy \eqref{eq:BVPs}. 
Thus $Ax=y$, and hence  $\ker(\phi)\subset\Ran(A)$.

By Proposition~\ref{prp:kerA}, $\ker(A)=\Span \set{(1,0,0)}$. Together with the above characterisation of $\Ran(A)$ this shows that $X=\Ran(A)\oplus\ker(A)$.
It remains to show that the norms  
$\snorm{\cdot}$ and $\norm{\cdot}$ are equivalent on $\ran(A)$. Let $x=(\v,\vv,\f)\in\ran(A)$. Then
  \eq{
  \norm{\v}_{\Lp[2]}
  &\leq 
  \sup_{-1\leq \xi\leq 0} \abs{\v(\xi)} 
  \leq \abs{\v(0)}+ \sup_{-1\leq \xi\leq 0}\int_{\xi}^0 \abs{\v'(r)}\,dr.
  }
Since  $x\in \ker(\phi)$,
$$|u(0)|\le \int_{-1}^0|v(\xi)|\,d\xi+\int_0^1(1-\xi)|w(\xi)|\, d\xi,$$
and hence 
$\norm{u}_{L^2}\lesssim \norm{u'}_{L^2}+\norm{v}_{L^2}+\norm{w}_{L^2}.$
It follows that $\snorm{x}\lesssim\norm{x}$. Since $\norm{x}\le\snorm{x}$ trivially for all $x\in X$, the proof is complete.  
\end{proof}

\begin{proof}[Proof of Theorem~\textup{\ref{thm:sggeneration}}]
  We show first that $A$ generates a $C_0$-semigroup on $X$. 
Let $x=(\v,\vv,\f)\in \Dom(A)$. A straightforward computation using the definition of $D(A)$ 
and the Cauchy-Schwarz inequality
shows  that
$$\begin{aligned}
\re\iprod{(A-I)x}{x}
\le
\norm{\vv}_{\Lp[2]}\norm{\v}_{\Lp[2]} -\norm{\f'}_{\Lp[2]}^2 -\norm{\v}_{H^1}^2 -\norm{\vv}_{\Lp[2]}^2 -\norm{\f}_{\Lp[2]}^2  ,
\end{aligned}$$
and it follows from the scalar inequality $ab\leq \frac{1}{2}(a^2+b^2)$ that 
$A-I$ is dissipative. Since $1\in\rho(A-I)$ by Propositions~\ref{thm:sggeneration} and \ref{prp:Aeigvals}, the Lumer-Phillips theorem shows that $A-I$ generates a $C_0$-semigroup of contractions on $X$, and hence $A$ generates a $C_0$-semigroup $\Sg$. It remains to show that $\Sg$ is uniformly bounded. 

By Proposition~\ref{prp:Xdecomp}, $X=X_0\oplus X_1$ topologically, where $X_0=\Ran(A)$ and $X_1=\ker(A)$ are closed $\Sg$-invariant subspaces of $X$. Moreover,  $X_0$ is a Hilbert space with the equivalent norm $\snorm{\cdot}$ induced by the inner product
  \eq{
\sprod{x}{y}
  = \iprod{\v'}{\vs'}_{\Lp[2]} + \iprod{\vv}{\vvs}_{\Lp[2]} + \iprod{\f}{\fs}_{\Lp[2]} .
  } 
where $x=(u,v,w)$ and $y=(f,g,h)$.
For $j=0,1$, let $A_j$ denote the restriction of $A$ to $X_j$, so that $D(A_j)=D(A)\cap X_j$. Furthermore, let $\Sg[T_j]$ denote the $C_0$-semigroup obtained by restricting $\Sg$ to $X_j$, noting that $A_j$ is the infinitesimal generator of $\Sg[T_j]$. Given $x=(\v,\vv,\f)\in \Dom(A_0)$, then similarly as in the proof of Proposition~\ref{prp:kerA} 
  \eq{
  \MoveEqLeft \re\, \sprod{A_0 x}{x}
  =-\norm{\f'}_{\Lp[2]}^2 \leq 0.
  }
Thus $A_0$ is dissipative. Note also that the splitting of the space coincides with the spectral decomposition for the eigenvalue 0. In particular, $\sigma(A_0)=\sigma(A)\backslash\{0\}$, and hence $1\in\rho(A_0)$. It follows from the Lumer-Phillips theorem that  $\Sg[T_0]$ is a contraction semigroup.
Moreover, for $x\in X_1$ and $t\ge0$, $T_1(t)x=x$.
Hence, given $x \in X$ and $x=x_0+x_1$ with $x_0\in X_0$, $x_1\in X_1$,  
\eq{
\label{eq:T_bdd}
\norm{T(t)x}
\lesssim \snorm{T_0(t)x_0} + \norm{x_1} 
\leq \snorm{x_0} + \norm{x_1}
\lesssim\|x\|,\quad t\ge0,
}
where the implicit constants are independent of $t$. It follows that $\Sg$ is  uniformly bounded, as required.
\end{proof}

\begin{rem}
\label{rem:A1spec}
The above proof shows that the restriction $A_0$ of $A$ to $\Ran(A)$ satisfies $\sigma(A_0)=\sigma(A)\backslash\{0\}$. In particular, $\sigma(A_0)\subset\CC_-$. 
\end{rem}

\section{Resolvent estimates}
\label{sec:res}

In this section we study the the behaviour of the resolvent operator $R(is,A)$ as $|s|\to\infty$, where $A$ is the generator of the semigroup $\Sg$ studied in Section~\ref{sec:sggen}. First, in Section~\ref{sec:upperbd}, we obtain an upper bound on the growth of $\|R(is,A)\|$ as $\abs{s}\to\infty$, and in Section~\ref{sec:lowerbd} we show that this estimate is optimal. These results will allow us to deduce sharp estimates on the rate of energy decay in Section~\ref{sec:energy} below.

\subsection{An upper bound}
\label{sec:upperbd}

The main result of this section is the following asymptotic upper bound on the operator norm of the resolvent operator.

\begin{thm}
  \label{thm:upperbd}
We have $\norm{R(is,A)}=O(|s|^{1/2})$ as $|s|\to\infty$.
\end{thm}

We begin with two technical estimates.

\begin{lem}
  \label{lem:Gfunest}
  There exists $c>0$ such that 
  \eq{
  \Abs{\sqrt{is}  \cos(s)\cosh(\sqlwi) +i \sin(s)\sinh(\sqlwi)}\geq c \exp\bigg(\frac{\abs{s}^{1/2}}{\sqrt{2}}\bigg),\quad \abs{s}\geq 2.
  }
\end{lem}

\begin{proof}
  Let $\gw\in\R$ with $\abs{\gw}\geq 2$.
By splitting $\sqrt{is}$ into real and imaginary parts and writing the hyperbolic functions in terms of exponentials, it is easy to see that
\eq{
\MoveEqLeft 
2\exp\bigg(-\frac{\abs{s}^{1/2}}{\sqrt{2}}\bigg)
\Abs{ \sqlwi \cos(\gw)\cosh(\sqlwi) + i   \sin(\gw)\sinh(\sqlwi)}\\
&\ge  \bigl|\sqlwi  \cos(\gw) \pm i \sin(\gw)\big|
- \exp\big(-\sqrt{2}|s|^{1/2}\big) \big|\sqlwi  \cos(\gw) \mp i \sin(\gw)\big|.
}
Note that
\eq{
\MoveEqLeft
\big|\sqlwi  \cos(\gw) \pm i \sin(\gw)\big|^2
=\frac{\abs{\gw}}{2} \cos^2(\gw) + \bigg({\frac{\abs{\gw}^{1/2}}{\sqrt{2}}}  \cos(\gw) \pm \sin(\gw)\bigg)^2.
}
Thus if $|s|^{1/2}\abs{\cos(\gw)}\geq 1/2$, then 
\begin{equation}
\label{eq:lower_est}
\big|\sqlwi  \cos(\gw) \pm i \sin(\gw)\big|
\geq \frac{1}{2\sqrt{2}}.
\end{equation}
Suppose, on the other hand, that $|s|^{1/2}\abs{\cos(\gw)}< 1/2$. Since $|s|\ge2$, it follows that $\abs{\cos(\gw)}\leq 1/(2\sqrt{2})$, and hence $\abs{\sin(\gw)}\geq 1/\sqrt{2}$. Thus
\eq{
\bigg|{\frac{\abs{\gw}^{1/2}}{\sqrt{2}}}  \cos(\gw) \pm \sin(\gw)\bigg|
\geq
\abs{\sin(\gw)} - {\frac{\abs{\gw}^{1/2}}{\sqrt{2}}}  \abs{\cos(\gw)}  
\geq 
\frac{1}{2\sqrt{2}},
}
so \eqref{eq:lower_est} holds in this case as well.
Since
\eq{
\exp\big(-\sqrt{2}|s|^{1/2}\big) \big|\sqlwi  \cos(\gw) \mp i \sin(\gw)\big|<\frac{1}{3}<\frac{1}{2\sqrt{2}}
}
whenever $\abs{\gw}\geq 2$, the result follows.
\end{proof}

\begin{lem}
  \label{lem:Vfunest}
There exists a constant $C\geq 0$ such that,
for all $f\in H^1(-1,0),$ $g\in L^2(-1,0)$, $s\in\R\setminus \set{0}$ and
 $\xi \in [-1,0]$,
  \eq{
  \Abs{\int_{-1}^\xi \sin(\gw(\xi-r))\big(i\gw \vs(r)+\vvs(r)\big)\,dr}
  &\leq C  \norm{\vs}_{H^1} + \norm{\vvs}_{\Lp[2]} , \\
  \Abs{\int_{-1}^\xi \cos(\gw(\xi-r))\big(i\gw \vs(r)+\vvs(r)\big)\,dr}
  &\leq C  \norm{\vs}_{H^1} + \norm{\vvs}_{\Lp[2]}.
  } 
\end{lem}

\begin{proof}
Using integration by parts it is easy to show that
$$
\begin{aligned}
\int_{-1}^\xi \sin(\gw(\xi-r))s\vs(r)\,dr&=\int_{-1}^\xi \big(1-\cos(\gw(\xi-r))\big)\vs'(r)\,dr\\
&\qquad\qquad +  \big(1-\cos(\gw(\xi+1))\big)\vs(-1).
\end{aligned}
$$
  Since the evaluation functional $f\mapsto f(-1)$ is continuous on $H^1(-1,0)$, it follows  that 
  $$
 \Abs{ \int_{-1}^\xi \sin(\gw(\xi-r))s\vs(r)\,dr}\lesssim \norm{f}_{H^1}.
  $$
  An analogous argument shows that 
  $$
   \Abs{ \int_{-1}^\xi \cos(\gw(\xi-r))s\vs(r)\,dr}\lesssim \norm{f}_{H^1},
$$
  and the result follows by an application of the Cauchy-Schwarz inequality to the terms involving $g$.
\end{proof}

\begin{proof}[Proof of Theorem~\textup{\ref{thm:upperbd}}]
  For $s\in\RR$ with $\abs{s}\geq 2$, let $y=(\vs,\vvs,\fs)\in X$ and define $x=(\v,\vv,\f)\in \Dom(A)$ by $x=R(is,A)y$. Then 
 \eq{
\norm{x}
=\left(  \norm{\v}_{H^1}^2 + \norm{\vv}_{\Lp[2]}^2 + \norm{\f}_{\Lp[2]}^2 \right)^{1/2}
\lesssim  \norm{\v}_{\Lp[2]} + \norm{\v'}_{\Lp[2]} + \norm{\vv}_{\Lp[2]} + \norm{\f}_{\Lp[2]}
} 
and since $(is-A)x=y$ implies that $\vv=is\v-\vs$, this becomes 
\eqn{
\label{eq:xest}
\norm{x}\lesssim \norm{s\v}_{\Lp[2]} + \norm{\v'}_{\Lp[2]} +  \norm{\f}_{\Lp[2]} + \norm{\vs}_{\Lp[2]} ,\quad |s|\ge2.
}
Thus the result will follow once we have established the following estimates:\begin{subequations}
  \label{eq:wavepartest}
   \eqn{
\label{eq:uest}  
\norm{s \v}_{\Lp[2]}
  &\lesssim \abs{s}^{1/2} \norm{\vs}_{H^1} +\abs{s}^{1/2}  \norm{\vvs}_{\Lp[2]} +  \norm{\fs}_{\Lp[2]},\\
  \label{eq:upest}
\norm{ \v'}_{\Lp[2]}
  &\lesssim \abs{s}^{1/2} \norm{\vs}_{H^1} +\abs{s}^{1/2}  \norm{\vvs}_{\Lp[2]} +  \norm{\fs}_{\Lp[2]},\\
  \label{eq:west} 
  \|w\|_{L^2}&\lesssim  \norm{\vs}_{H^1} +\norm{\vvs}_{\Lp[2]} +  \norm{\fs}_{\Lp[2]}
  }
\end{subequations}
for all $s\in\RR$ with 
$\abs{s}\geq 2$. Indeed, by \eqref{eq:xest} these estimates imply that $\|x\|\lesssim |s|^{1/2}\|y\|$ for $|s|\ge2$.
Since $y\in X$ was arbitrary, it then follows that $\norm{R(is,A)}=O(\abs{s}^{1/2})$ as $|s|\to\infty$, as required. 

Let $s\in \R$ be such that $\abs{s}\geq 2$. We begin by deriving formulas for the components $u$ and $w$ of the vector $x=(\v,\vv,\f)$. Note first that the equation $(is-A)x=y$ is equivalent to the following system of boundary value problems: 
\begin{subequations}
  \eqn{
  \label{eq:ResBVP1}
  \hspace{1cm} \v''(\xi)&=-s^2\v(\xi)-is\vs(\xi)-\vvs(\xi), \hspace{-0.2cm} & \hspace{-2cm} \xi\in(-1,0), \hspace{0.8cm}\\
  \hspace{1cm}    \vv(\xi)&=is\v(\xi)-\vs(\xi), \hspace{-0.2cm}  & \hspace{-3.3cm} \xi\in(-1,0), \hspace{0.8cm}\\
  \label{eq:ResBVP3}
  \hspace{1cm}     \f''(\xi)&=is\f(\xi)-\fs(\xi), \hspace{-0.2cm} & \hspace{-3.3cm} \xi\in(0,1), \hspace{0.8cm} \\
  & \hspace{-0.9cm}  \v'(-1)=  \f(1)=0, ~
  \vv(0)=\f(0), ~
  \v'(0)=\f'(0).
  }
\end{subequations}
Let
  \begin{equation*}
    \label{eq:Vdef}
    \Ivs[s]{\xi} =  \frac{1}{s}\int_{-1}^\xi \sin(s (\xi-r))\big(is \vs(r)+\vvs(r)\big)\,dr,\quad \xi\in[-1,0],
  \end{equation*}
  and note that 
    \begin{equation*}
    \label{eq:Uder}
    \Ivc[s]{\xi} = \int_{-1}^\xi \cos(s (\xi-r))\big(is \vs(r)+\vvs(r)\big)\,dr,\quad \xi\in[-1,0].
    \end{equation*}
The
general solution of the differential equation~\eqref{eq:ResBVP1}
  with the boundary condition $\v'(-1)=0$ can be written as 
    \eqn{
    \label{eq:usol}
\v(\xi) &=  \cv[s] \cos(s(\xi+1))-  \Ivs{\xi},\quad \xi\in[-1,0],
}
  where $\cv[s]\in \C$. In particular,
\eqn{
\label{eq:uder}
\v'(\xi) &=  -s\cv[s]\sin(s(\xi+1)) -  \Ivc{\xi} ,\quad \xi\in[-1,0].
  } 
Furthermore, let  
\begin{equation}
    \label{eq:Wdef}
    \Ifs{\xi} = - \frac{1}{\sqlwi}\int_\xi^1 \sinh(\sqlwi (r-\xi)) \fs(r)dr,\quad \xi\in[0,1],
\end{equation}
noting that
\begin{equation}
    \label{eq:Wder}
 \Ifc{\xi} = \int_\xi^1 \cosh(\sqlwi (r-\xi)) \fs(r)dr, \quad \xi\in[0,1].
 \end{equation}
 The general solution of the differential equation~\eqref{eq:ResBVP3}
  with the boundary condition $\f(1)=0$ is given by
  \eqn{
  \label{eq:wsol}
  \f(\xi)&=   -\cf[s]\sinh(\sqlwi(1-\xi))\ +  \Ifs{\xi} ,\quad \xi\in[0,1],
  }
  where $\cf[s]\in \C$. In particular,
  \eq{
  \f'(\xi)&=  \sqlwi\,\cf[s] \cosh(\sqlwi(1-\xi))+ \Ifc{\xi},\quad \xi\in[0,1].
  }
It remains to determine the constants $\cv[s]$ and $\cf[s]$. 
Using~\eqref{eq:usol} and~\eqref{eq:wsol}, the coupling condition $\f(0)=\vv(0)=is \v(0)-\vs(0)$ is easily seen to be equivalent to
\eq{
      is \cv[s]\cos(s)+   \cf[s]\sinh(\sqlwi)  =  \vs(0) + is\Ivs{0} +  \Ifs{0},
}
and similarly the condition $\f'(0)=\v'(0)$ is equivalent to
\eq{
  s\cv[s]\sin(s) +\sqlwi \,\cf[s] \cosh(\sqlwi) = - \Ivc{0}- \Ifc{0}.
}
These two equations can be written in the form
\eqn{
\label{eq:matrix}
\MoveEqLeft 
\pmat{
is  \cos(s) &
 \sinh(\sqlwi)\\[1.5ex] 
s\sin(s) & \sqlwi\cosh(\sqlwi) }
\pmat{\cv[s]\\[.5ex]\cf[s]}
=
\pmat{ \vs(0) +  is\Ivs{0} + \Ifs{0} \\[.5ex]
-\Ivc{0} - \Ifc{0}
}.
}
Let
\eq{
M(s) =
\pmat{
is  \cos(s) &
 \sinh(\sqlwi)\\[1.5ex] 
s\sin(s) & \sqlwi\cosh(\sqlwi) }.  
    }
Then
\eqn{
\label{eq:det}
\det M(s)
= (is)^{3/2}  \cos(s)\cosh(\sqlwi) -s   \sin(s)\sinh(\sqlwi),
}
and Theorem~\ref{thm:sggeneration} implies that $\det M(s)\neq 0$ for all $s\in\RR$ with  $s\neq 0$. In particular, $M(s)$ is invertible for $\abs{s}\geq 2$.
Solving~\eqref{eq:matrix} finally gives
\begin{equation}
\begin{aligned}
\label{eq:a}
\cv[s] &= \frac{\sqlwi\cosh(\sqlwi)}{\det M(s)} \big(\vs(0) + is\Ivs[s]{0}+ \Ifs[s]{0}\big)  \\
&\qquad\qquad
+ \frac{\sinh(\sqlwi)}{\det M(s)} \big( \Ivc[s]{0} + \Ifc[s]{0}\big), 
\end{aligned}
\end{equation}
\vspace{-2ex}
\begin{equation}
\begin{aligned}
\label{eq:b}
\cf[s] &= 
 -\frac{s \sin(s)}{\det M(s)}  \big(  \vs(0) +  is\Ivs[s]{0}+W_s(0)\big)\\
  &\qquad\qquad
-  \frac{is\cos(\gw)}{\det M(s)} \big( \Ivc[s]{0}+\Ifc[s]{0}\big).
\end{aligned}
\end{equation}

We  now establish the estimates~\eqref{eq:uest} and~\eqref{eq:upest}. 
By \eqref{eq:usol} and \eqref{eq:uder},
\eqn{
\label{eq:uest2}
|su(\xi)|\le |sa(s)|+|sU_s(\xi)|\quad \mbox{and}\quad|u'(\xi)|\le |sa(s)|+|U'_s(\xi)|
}
for all $\xi\in[-1,0]$, and Lemma~\ref{lem:Vfunest} gives
\eqn{
\label{eq:Uest}
\sup_{\xi\in[-1,0]} \abs{s\Ivs[s]{\xi}} 
,\;
\sup_{\xi\in[-1,0]}\abs{\Ivc[s]{\xi}}
\lesssim \norm{\vs}_{H^1} + \norm{\vvs}_{\Lp[2]}.
} 
It therefore  remains to estimate $|sa(s)|$ for $|s|\ge2$. 
A simple estimate based on  Lemma~\ref{lem:Gfunest} shows that
\eqn{
\label{eq:detest}
\bigg|\frac{{s \cosh(r\sqlwi)}}{{\det M(s)}}\bigg|\lesssim 1\quad\mbox{and}\quad\bigg|\frac{{s \sinh(r\sqlwi)}}{{\det M(s)}}\bigg|\lesssim 1
}
for $\abs{r}\leq 1$ and $|s|\ge2$. 
Using these together with \eqref{eq:Wdef}, \eqref{eq:Wder} and \eqref{eq:Uest} in \eqref{eq:a} gives
\eq{
\abs{s\cv[s]}
&\leq \biggl| \sqlwi \,\frac{s\cosh(\sqlwi)}{\det M(s)} \big(\vs(0) + is\Ivs[s]{0} \big)
+ \frac{s\sinh(\sqlwi)}{\det M(s)} \Ivc[s]{0}\biggr| \\
& \qquad\qquad + \biggl|\int_0^1 \frac{ s\sinh(\sqlwi(1-r)) }{\det M(\gw)} \fs(r)\,dr \biggr|\\
&\lesssim \abs{s}^{1/2} \norm{\vs}_{H^1} +\abs{s}^{1/2} \norm{\vvs}_{\Lp[2]} 
+ \norm{\fs}_{\Lp[2]}
}
for $|s|\ge2$. Thus \eqref{eq:uest} and and~\eqref{eq:upest} follow from \eqref{eq:uest2} and \eqref{eq:Uest}.

We now turn to proving \eqref{eq:west}. 
By \eqref{eq:wsol},~\eqref{eq:det}  and \eqref{eq:b}, we have that
\eq{
\f(\xi)  
&= \frac{s \sinh(\sqlwi(1-\xi))}{\det M(s)}  \Bigl(\sin(s)  \big(  \vs(0) +  is\Ivs[s]{0}\big) + i \cos(\gw)  \Ivc[s]{0}\Bigr)\\
&\qquad+  \frac{s\sin(s)}{\det M(s)} \Bigl( \sinh(\sqlwi(1-\xi))    \Ifs[s]{0} - \sinh(\sqlwi)   \Ifs[s]{\xi} \Bigr) \\
&\qquad+ \frac{is\cos(\gw)}{\det M(s)} \Bigl( \sinh(\sqlwi(1-\xi)) \Ifc[s]{0} 
+ \sqlwi \cosh(\sqlwi) \Ifs[s]{\xi} \Bigr)
} 
for all $\xi\in[0,1]$. Let $w_1(\xi)$, $w_2(\xi)$, $w_3(\xi)$, respectively, denote the three terms 
on the right-hand side of this equation.
We estimate $w_1$, $w_2$, $w_3$ in turn. From \eqref{eq:Uest} and \eqref{eq:detest} it is clear that $$|w_1(\xi)|\lesssim\|f\|_{H^1}+\|g\|_{L^2},\quad \xi\in[0,1].$$
Turning  to $w_2$, it follows from an elementary calculation using \eqref{eq:Wdef} that
$$
\begin{aligned}
2\sqlwi&\Big( \sinh(\sqlwi(1-\xi))   \Ifs{0} -
\sinh(\sqlwi)   \Ifs{\xi}\Big) \\
&=   \int_0^1  \cosh(\sqlwi (1-\xi-r)) \fs(r)\,dr - \int_0^\xi  \cosh(\sqlwi (1-\xi+r)) \fs(r)\,dr\\
&\qquad\qquad -   \int_\xi^1 \cosh(\sqlwi(1+\xi-r))  \fs(r)\,dr
\end{aligned}
$$
for all $\xi\in[0,1]$.
Hence a simple estimate using \eqref{eq:detest} gives 
$$|s|^{1/2}|w_2(\xi)|\lesssim \norm{h}_{L^2},\quad\xi\in[0,1].$$
An analogous argument shows that 
$$|w_3(\xi)|\lesssim \norm{h}_{L^2},\quad\xi\in[0,1].$$
Combining these estimates gives  \eqref{eq:west}, thus completing the proof.
\end{proof}

\subsection{Optimality of the upper bound}
\label{sec:lowerbd}

The following result shows that the spectrum $\sigma(A)$ of $A$ contains sequences of eigenvalues whose real parts tend to $\pm\infty$ and whose imaginary parts approach zero at a  rate which  shows that the resolvent estimate in Theorem~{\ref{thm:upperbd}} cannot be improved.

\begin{thm}
\label{thm:opt}
There exist sequences $(\gl_n^\pm)$ in $\sigma(A)$ and a constant $C>0$  such that $\im \gl_n^\pm\sim \pm n\pi$ as $n\to\infty$ and
\eqn{
\label{eq:sequ}
-\frac{C}{|\im\gl_n^\pm|^{1/2}}\le\re\gl_n^\pm<0,\quad n\ge0.
}
In particular,  
 $$\limsup_{|s|\to\infty}|s|^{-1/2}\|R(is,A)\|>0.$$
\end{thm}

\begin{proof}
Consider the meromorphic functions $F$ and $G$ defined  by 
$F(\lambda)=\coth(\lambda)$ and $G(\lambda)=\tanh({\sql})/\sql$,
 where the square root is defined with a branch cut along the negative real axis. By Theorem~\ref{thm:sggeneration}, the roots of the function $F+G$ are eigenvalues of $A$. We use Rouch\'{e}'s theorem to determine the approximate location of such roots; see for instance \cite{DFMP11} and \cite{ZhaZua04} for related arguments. 

The function $F$ is $\pi i$-periodic, with simple poles at $\lambda=n\pi i$ and simple zeros at $\lambda=\lambda_n$, where $\lambda_n=(n+\smash{\frac{1}{2}})\pi i$, $n\in\ZZ$.  Since $F'(\lambda_n)=1$ for all $n\in\ZZ$, a simple argument using a Taylor expansion shows that 
$|F(\lambda)|\ge\smash{\frac{1}{2}}|\gl-\gl_n|$
 provided  $|\gl-\gl_n|$ is sufficiently small. For $n\in\ZZ$, let $r_n=|n+\frac{1}{2}|^{-1/2}$ and $\Omega_{n}=\{\lambda\in\CC:|\lambda-\lambda_n|< 2r_n\}$. 
Then for all sufficiently large $\abs{n}$ we have that $|F(\lambda)|\ge r_n$ for $\lambda\in\partial\Omega_n$ and both $F$ and $G$ are holomorphic in a region containing the closure of $\Omega_n$. 
Simple estimates show that 
$$\sup_{\lambda\in\partial\Omega_n}|G(\lambda)|\sim\frac{r_n}{\sqrt{\pi}},\quad n\to\pm\infty, $$
and hence  $|G(\lambda)|<|F(\lambda)|$ for $\lambda\in\partial\Omega_n$ with $|n|$ sufficiently large. By Rouch\'{e}'s theorem,  $F$ and $F+G$ have the same number of zeros inside $\Omega_n$ for $n\in\ZZ$ as above, namely one. Theorem~\ref{thm:sggeneration} implies that $\sigma(A)\subset\CC_-\cup\{0\}$, so \eqref{eq:sequ} follows. The final statement is a simple consequence of the fact that $\|R(is,A)\|\ge\dist(is,\sigma(A))^{-1}$ for all $s\ne0$.
\end{proof}

\section{Energy decay}
\label{sec:energy}

Theorem~\ref{thm:energy} below gives a quantified estimate for the rate of energy decay of classical solutions to \eqref{eq:problem}. The proof relies on the following abstract result; see \cite[Theorem~2.4]{BT10}.

\begin{thm}
\label{thm:BT}
Let $\Sg$ be a uniformly bounded $C_0$-semigroup on a Hilbert space $X$. Let $A$ be the generator of $\Sg$ and suppose that $\sigma(A)\cap i\RR=\emptyset$. Then, for any constant $\alpha>0$, the following conditions are equivalent:
\begin{itemize}
    \setlength{\itemsep}{.8ex}
  \item[\textup{(i)}] $\norm{R(is,A)}=O(|s|^\alpha)$ as $|s|\to\infty$;
  \item[\textup{(ii)}] $\norm{T(t)A^{-1}}=O(t^{-1/\alpha})$ as $t\to\infty$;
  \item[\textup{(iii)}] $\norm{T(t)x}=o(t^{-1/\alpha})$ as $t\to\infty$ for all $x\in D(A)$.
\end{itemize}
\end{thm}

The following result is the main result of this section. 

\begin{thm}
\label{thm:energy}
If $x\in D(A)$, then $E_x(t)=o( t^{-4})$ as $t\to\infty$.
\end{thm}

\begin{proof}
Using the notation introduced in the proof of Theorem~\ref{thm:sggeneration}, 
$$E_x(t)=\frac{1}{2}\snorm{T(t)x}^2=\frac{1}{2}\snorm{T_0(t)x_0}^2,\quad t\ge0,$$
where $x=x_0+x_1$ with $x_0\in X_0$ and $x_1\in X_1$. For $y\in X_0$, it follows from the equivalence of the norms $\|\cdot\|$ and $\snorm{\cdot}$ on $X_0$ that 
$$\snorm{R(is,A_0)y}\le\norm{R(is,A)y}\lesssim \norm{R(is,A)}\snorm{y},\quad s\ne0,$$
and hence by Theorem~\ref{thm:upperbd}
$$\|R(is,A_0)\|=O\big(|s|^{1/2}\big),\quad |s|\to\infty,$$ 
where the operator norm is taken with respect to the $\snorm{\cdot}$-norm on $X_0$. 
Since $\sigma(A_0)\cap i\RR=\emptyset$ by Remark~\ref{rem:A1spec}, it follows from Theorem~\ref{thm:BT} that $\snorm{T_0(t)x_0}=o(t^{-2})$ as $t\to\infty$, and the result follows.
\end{proof}

\begin{rem}
\begin{enumerate}[(a)]
\item The rate $t^{-4}$ in Theorem~\ref{thm:energy} is optimal in the sense that, given any positive function $r$ satisfying $r(t)=o(t^{-4})$ as $t\to\infty$, there exists $x\in D(A)$ such that $E_x(t)\ne o(r(t))$ as $t\to\infty$. This follows from Theorem~\ref{thm:opt} and  the uniform boundedness principle together with  \cite[Proposition~1.3]{BD08}; see also \cite[Theorem~4.4.14]{ABHN11}.
\item 
If we denote by $A_0$ the restriction of $A$ to $\ran(A)$ as in the proof of Theorem~\ref{thm:sggeneration},
then Theorem~\ref{thm:opt}  implies that
the spectral bound of $A_0$ satisfies $s(A_0)=0$.
Since $\Sg[T_0]$ is uniformly bounded, its growth bound is equal to $\omega_0(T_0)=0$.
Suppose there exists a positive function $r$ such that $r(t)\to0$ as $t\to\infty$ and for which $E_x(t)=O(r(t))$ as $t\to\infty$ for all $x\in X$. Then by the uniform boundedness principle $\|T_0(t)\|=O(r(t)^{1/2})$ as $t\to\infty$. This implies that $\omega_0(T_0)<0$,  which is a  contradiction.
Hence there is no hope of finding a rate of energy decay which is valid for all $x\in X$; see \cite[Lemma~3.1.7]{vN96} for a related result.
\item The proof of \cite[Theorem~2.4]{BT10} can easily be extended, simply by using the semigroup property, to show that condition (i) in Theorem~\ref{thm:BT} is equivalent to having $\|T(t)x\|=o(\smash{t^{-k/\alpha}})$ as $t\to\infty$ for all $x\in D(A^k)$ and all integers $k\ge1$. Thus the proof of Theorem~\ref{thm:energy} shows that $E_x(t)=o(t^{-4k})$ as $t\to\infty$ for all $x\in D(A^k)$, $k\ge1$, which is to say that smoother orbits have faster energy decay. 
\end{enumerate}
\end{rem}

\section{The case of Dirichlet boundary conditions}
\label{sec:Dirichlet}

As mentioned in Section~\ref{sec:Intro}, the above approach can be modified straightforwardly to deal with the 
model
\eq{
\left\{
\begin{array}{ll@\qquad r}
  \multicolumn{2}{l}{\v_{tt}(\xi,t) =\v_{\xi\xi}(\xi,t),} 
  & \xi\in(-1,0), ~t>0, \\[1ex]
  \multicolumn{2}{l}{\f_{t}(\xi,t) = \f_{\xi\xi}(\xi,t),} &   \xi\in(0,1), ~t>0,\\[1.5ex]
    \v(-1,t)=0, & \f(1,t)=0, & t>0,\\[1ex]
    \v_t(0,t)=\f(0,t), & \v_\xi(0,t)=\f_\xi(0,t), & t>0,\\[1.5ex]
    \v(\xi,0)=\v(\xi), & \v_t(\xi,0)=v(\xi), &  \xi\in(-1,0),
  \\[1ex]
  \multicolumn{2}{l}{\f(\xi,0)=\f(\xi),} & \xi\in(0,1),
\end{array}
\right.
}
where Neumann boundary condition $u_\xi(-1,t)=0$ appearing in \eqref{eq:problem} has been 
replaced by the Dirichlet boundary condition $u(-1,t)=0$ for all $t>0$. Indeed, in this case it is easy to show by means of the Lumer-Phillips theorem that the operator $A$ with the new domain
$$\begin{aligned}
D(A)=\big\{(u,v,w)\in Y: \;& u(-1)=v(-1)=w(1)=0, \\
&\qquad v(0)=w(0), ~u'(0)=w'(0)\big\}
\end{aligned}$$
generates a contraction semigroup on the Hilbert space
$$Z=\big\{(u,v,w)\in X: u(-1)=0\big\}.$$ Moreover, the spectrum of $A$ again consists only of isolated eigenvalues of finite multiplicity, and in fact
$$\sigma(A)=\big\{\lambda\in\CC\backslash\{0\}: \sql \sinh(\lambda)\cosh(\sql)+ \cosh(\lambda)\sinh(\sql)=0\big\}.$$
In particular $\sigma(A)\subset\CC_-$. Arguments completely analogous to those presented in Section~\ref{sec:res} again lead to the optimal resolvent bound 
$\|R(is,A)\|=O(|s|^{1/2})$ as $|s|\to\infty.$
 This time Theorem~\ref{thm:BT} can be applied directly to $\Sg$, and the result 
 $E_x(t)=o(t^{-4})$ as $t\to\infty$
 for all $x\in D(A)$ follows at once from the fact that, by the Poincar\'{e} inequality, the energy defines an equivalent norm on $Z$.

These results are obtained in \cite{ZhaZua04} for a  similar problem involving an alternative coupling condition and the corresponding 
 more restrictive 
choice of  $D(A)$. The technique used in \cite{ZhaZua04} relies on the theory of Riesz spectral operators and requires a very detailed
spectral analysis.
The present approach based on Theorem~\ref{thm:BT}  is rather simpler and in particular requires no particular knowledge about those eigenvalues which do not approach the imaginary axis, nor  any knowledge about any eigenfunctions.

\end{document}